\documentclass[a4paper, 11 pt, twoside]{amsart}
\usepackage{srollens-en}[2011/03/03]
\usepackage[left = 3.5cm, right = 3.5cm, headsep = 6mm,
footskip = 10mm, top = 35mm, bottom = 35mm, footnotesep=5mm, headheight =
2cm]{geometry}

\usepackage{tikz, tikz-cd}

\usepackage[unicode,bookmarks, pdftex, backref = page]{hyperref}
\hypersetup{colorlinks=true,citecolor=NavyBlue,linkcolor=NavyBlue,urlcolor=Orange, pdfpagemode=UseNone, breaklinks=true}

\usepackage{enumitem}
\setlist[description]{leftmargin=0cm,  labelindent=\parindent}

\newcommand{\Wedge}{\textstyle\bigwedge}

\newcommand{\Vol}{\hbox{Vol}}
\DeclareFontFamily{OT1}{rsfs}{}
\DeclareFontShape{OT1}{rsfs}{n}{it}{<-> rsfs10}{}
\DeclareMathAlphabet{\curly}{OT1}{rsfs}{n}{it}

\DeclareMathOperator{\Def}{Def}
\DeclareMathOperator{\Grass}{Grass}

\title[$\del\delbar$-complex symplectic manifolds]{$\del\delbar$-complex symplectic and Calabi-Yau manifolds: Albanese map, deformations and period maps}

\author{Ben Anthes}
\address{Ben Anthes\\FB 12/Mathematik und Informatik\\
Philipps-Universit\"at Marburg\\
Hans-Meerwein-Str. 6\\
35032 Marburg\\
Germany}
\email{anthes@mathematik.uni-marburg.de}

\author{Andrea Cattaneo}
\address{Andrea Cattaneo\\
Institut Camille Jordan\\
Universit\'e Claude Bernard Lyon 1\\
43, boulevard du 11 novembre 1918\\
69100 Villeurbanne, France}
\email{cattaneo@math.univ-lyon1.fr}

\author{S\"onke Rollenske}
\address{S\"onke Rollenske\\FB 12/Mathematik und Informatik\\
Philipps-Universit\"at Marburg\\
Hans-Meerwein-Str. 6\\
35032 Marburg\\
Germany}
\email{rollenske@mathematik.uni-marburg.de}

\author{Adriano Tomassini}
\address{Adriano Tomassini\\
Dipartimento di Scienze Matematiche, Fisiche e Informatiche\\
Plesso Matematico e Informatico\\
Universit\`a di Parma\\
Parco Area delle Scienze, 53/A, 43124\\
Parma, Italy}
\email{adriano.tomassini@unipr.it}

\begin{document}

\begin{abstract}
Let $X$ be a compact complex manifold with trivial canonical bundle and  satisfying the $\del\delbar$-Lemma. We show that the Kuranishi space of $X$ is a smooth universal deformation and that small deformations enjoy the same properties as $X$. If, in addition, $X$ admits a complex symplectic form, then the local Torelli theorem holds and we obtain some information about the period map.

We clarify the structure of such manifolds a little by showing that the Albanese map is a surjective submersion.
\end{abstract}

\maketitle
\setcounter{tocdepth}{1}
\tableofcontents

\section{Introduction}
Compact K\"ahler manifolds with trivial canonical bundle have attracted a large interest over the past decades, lying at the crossroads of differential geometry, algebraic geometry and mathematical physics.

In this paper, we relax the K\"ahler condition and only assume that the $\del\delbar$-Lemma holds, a condition which was introduced in \cite{DGMS}. It turns out that on a compact complex manifold $X$ with trivial canonical bundle the $\del\delbar$-Lemma guarantees that  $X$ has a smooth universal deformation. In addition,  all sufficiently small deformations are again of this type (Section \ref{sect: unobstructed defs}).
Note  that  the triviality of the canonical bundle in itself is not strong enough: both properties are known to fail without the $\del\delbar$-Lemma (see \cite{Ueno80, rollenske11} for examples).

A key result for compact K\"ahler manifolds with trivial canonical bundle is the Beauville--Bogomolov decomposition, which says that every such manifold is (up to a finite cover) a product of manifolds of three types:  complex tori, Calabi--Yau manifolds and irreducible holomorphic symplectic manifolds. In Section \ref{sect: albanese} we prove a weaker analogue of this and show by example that the decomposition theorem does not hold assuming only the $\del\delbar$-Lemma.

Of the above three types, irreducible holomorphic symplectic manifolds have the richest general theory and we study the following generalisation of this class in more detail.
\begin{defin}\label{def: simple}
A \emph{$\del\delbar$-complex symplectic manifold} is a pair $(X, \sigma)$ where $X$ is a compact complex manifold satisfying the $\del\delbar$-Lemma and $\sigma\in H^0(X, \Omega^2_X)$ is a $d$-closed holomorphic symplectic form. We say $X$ is \emph{simple} if $\sigma$ is unique up to scalars, that is, $h^{2,0}(X) = 1$.
\end{defin}
By \cite{CattaneoTomassini16}, on a $\del\delbar$-complex symplectic manifold $X$ one can define the  Beauville--Bogomolov--Fujiki quadratic form $q_\sigma$ on $H^2_{dR}(X, \IC)$ (see Definition \ref{def: bbf-form}). Assuming that $X$ is simple, we show that this quadratic form behaves very much like in the irreducible holomorphic symplectic case. For example, in this case, the period map identifies the universal deformation space with an open subset of the quadric defined by $q_\sigma$ (Section \ref{sec: local torelli}).

In the non-simple case, the local Torelli theorem still turns out to hold, but the relation between the quadratic form and the deformation space breaks down, as we show in Section \ref{sect: non-primitive periods}.

Throughout the paper we mention open questions about the only partially explored class of $\del\delbar$-complex symplectic manifolds.

\subsection*{Acknowledgments}
Ben Anthes and S\"onke Rollenske gratefully acknowledge support from the DFG via the Emmy Noether program.
We enjoyed discussions on parts of this project with Andreas Krug and we thank Nicolas Perrin for some discussions about Gra{\ss}mannians. Adriano Tomassini is granted with a research fellowship by Istituto Nazionale di Alta Matematica INdAM, and is supported by the Project PRIN ``Variet\`a reali e complesse: geometria, topologia e analisi armonica'', by the Project FIRB ``Geometria Differenziale e Teoria Geometrica delle Funzioni'', and by GNSAGA of INdAM. Andrea Cattaneo is supported by Labex MiLyon and by GNSAGA of INdAM.

Andrea Cattaneo and Adriano Tomassini are grateful to Ben Anthes and S\"onke Rollenske for their kind hospitality at Marburg University during the preparation of the paper.

\subsection*{Notation}
Throughout this article, we work with complex manifolds.
If $X$ is a compact complex manifold, we consider the double complex $(\ka^{\ast,\ast}(X),\del,\delbar)$ of smooth complex valued differential forms on $X$, where $d = \del+\delbar$ is the usual decomposition.

Recall that $X$ is said to satisfy the \emph{$\del\delbar$-Lemma} if 
\[\ker \del \cap \ker \delbar \cap (\im \del + \im \delbar) = \im \del \delbar,\]
see \cite{DGMS} or \cite{Angella14} for further discussion.

\section{Structure of the Albanese map}\label{sect: albanese}
In this section we study the structure of the Albanese map of compact complex  manifolds with trivial canonical bundle and satisfying the $\del\delbar$-Lemma. This line of inquiry is inspired by work of  Matsushima \cite{MR0364670}, Lichnerowicz \cite{MR0253253} and Kawamata \cite[Thm.\ 24]{Kawamata81}.

\begin{lem}\label{lem: wedge omegabar}
Let $X$ be a compact complex manifold of dimension $n$ satisfying the $\del\delbar$-Lemma. Assume $\Omega$ is a nowhere vanishing holomorphic $n$-form. Then the linear maps
\begin{gather*}
  \Lambda_{\Omega}\colon H^{0,q}_{\delbar}(X) \to H^{n,q}_{\delbar}(X), \quad [\alpha] \mapsto [\Omega \wedge \alpha]\\
  \Lambda_{\bar \Omega}\colon H^{p,0}_{\delbar}(X) \to H^{p,n}_{\delbar}(X), \quad [\alpha] \mapsto [\bar\Omega \wedge \alpha]
\end{gather*}
are both isomorphisms.
\end{lem}
\begin{proof}
If we denote by $\ka^{p,q}(X)$ the space of smooth forms of type $(p,q)$ and by $\ka^q(X, \Omega_X^p)$ the space of smooth forms of type $(0,q)$ with values in $\Omega^p_X$, the bundle of holomorphic $p$-forms, then the natural map
\[\ka^*(X, \Omega_X^p) \to \ka^{p,*}(X), \quad \phi\tensor \alpha \mapsto \phi\wedge \alpha\]
is an isomorphism. Thus, without using the $\del\delbar$-Lemma, we recognise $\Lambda_\Omega$ as the composition of isomorphisms
\[ \begin{tikzcd} H^{0,q}_{\delbar}(X) \rar{\isom} & H^q(X, \ko_X)\rar{\cdot \Omega} & H^q(X, \Omega^n_X) \rar {\isom} &  H^{n,q}_{\delbar}(X),
\end{tikzcd}\]
where the middle arrow is induced by the trivialising section $\cdot \Omega\colon \ko_X\to \Omega_X^n$.

We now use the $\del\delbar$-Lemma to the extent that on $X$ there is a decomposition
\[H^k_{dR}(X, \IC) = \bigoplus_{p+q=k} H^{p,q}_{\delbar}(X)\]
with the additional symmetry $H^{p,q}_{\delbar}(X)=\overline{H^{q,p}_{\delbar}(X)}$ under complex conjugation in the complex vector space $H^k_{dR}(X, \IC)$. 
Using those representatives for Dolbeault cohomology, we have $\Lambda_{\bar\Omega} = \overline{\Lambda_\Omega}$ and the claim follows.
\end{proof}

\begin{rem}
If in the situation of the Lemma the manifold $X$ is $\del\delbar$-complex symplectic with symplectic form $\sigma$ and $\Omega =\sigma^n$, then $\Lambda_\Omega = \Lambda_\sigma^n$. From this one can deduce that, for example, the map  $\Lambda_\sigma\colon H^{0,q}(X) \to H^{2, q}(X) $ is injective.
\end{rem}

\begin{exam}\label{exam: kodaira fibration lambda}
Maybe somewhat suprisingly, Lemma~\ref{lem: wedge omegabar} may fail without the additional symmetry provided by complex conjugation, even if the Fr\"olicher spectral sequence degenerates at $E_1$. As an example consider a Kodaira surface as described in Section~\ref{ex: Kodaira surface}. Then $\Omega = \omega_1 \wedge \omega_2$ is a holomorphic volume form and $\omega_1$ is a generator for $H^{1,0}_{\delbar} (X)\subset H^2_{dR}(X, \IC)$, that is, a holmorphic $1$-form.
But 
\[\Lambda_{\bar \Omega} \omega_1 = \omega_1\wedge \bar \Omega = \omega_1\wedge \bar \omega_1\wedge \bar \omega_2 = (\delbar \omega_2) \wedge \bar \omega_2=\delbar( \omega_2 \wedge \bar \omega_2)\]
is trivial in Dolbeault cohomology.
\end{exam}

\begin{prop}\label{prop: contraction non-degenerate}
Let $X$ be a compact complex manifold of dimension $n$ with trivial canonical bundle and satisfying the $\del\delbar$-Lemma. Then the evaluation map
\[b\colon H^0(X, \kt_X)\times H^0(X, \Omega_X^1)\to H^0(X, \ko_X) = \IC\]
is non-degenerate.
\end{prop}
\begin{proof}
Let $\Omega$ be a trivialising section of $\Omega_X^n$ such that $\int_X \bar \Omega\wedge \Omega=1$. By Lemma~\ref{lem: wedge omegabar} (and the first part of the proof) the vertical arrows in the diagram
\[
\begin{tikzcd}
 H^0(X, \kt_X)\times H^0(X, \Omega^1_X) \dar{(-\tensor \Omega) \times \Lambda_{\bar\Omega}} \rar{b} & H^0(X, \ko_X)\dar{\Lambda_\Omega\Lambda_{\bar\Omega}}\rar[equal] & \IC \dar[equal]\\
 H^0(X, \kt_X\tensor \Omega_X^n)\times H^n(X, \Omega_X^1) \rar{\text{S.D.}} & H^n(X, \Omega_X^n)\rar{\int_X } & \IC
\end{tikzcd}
\]
are isomorphisms. To make the diagram commute, the pairing in the second row has to be evaluation on the bundle part and wedge on the form part, followed by integration over $X$. This is exactly the definition of the Serre-Duality pairing (see e.g.\ \cite[Ch.~4.1]{Huybrechts}), which is non-degenerate. Hence, our claim follows.
\end{proof}

\begin{rem}\label{rem: Albanese}
Recall from \cite{Ueno74} that for a compact complex manifold the Albanese torus is defined as 
\[\Alb(X) = H^0(X, d\ko_X)^*/H, \]
where $H^0(X, d\ko_X)\subset H^1_{\text{dR}}(X, \IC)$ is the space of closed holomorphic $1$-forms and $H$ is the closure of the image of $H_1(X, \IZ)$ in $H^0(X, d\ko_X)^*$. Fixing a base point in $X$, integration over paths gives the Albanese map $\alpha\colon X\to \Alb (X)$, which is universal for pointed maps to complex tori.

If $X$ satisfies the $\del\delbar$-Lemma, then every holomorphic form is automatically closed and the image of $H_1(X, \IZ)$ in $H^0(X, \Omega^1_X)^*$ is indeed a cocompact lattice, so that $\Alb(X)$ is a complex torus of dimension $h^1(X, \ko_X)=\frac 12 b_1(X)$. 
\end{rem}

\begin{thm}\label{thm: Albanese}
Let $X$ be a compact complex manifold with trivial canonical bundle and satisfying the $\del\delbar$-Lemma. Let $G = \Aut_{\ko}(X)_0$ be the connected component of the identity of the holomorphic automorphism group of $X$ and $\alpha \colon X \to \Alb(X)$ the Albanese map with respect to a base-point $x_0$. Then the following hold:
\begin{enumerate}
\item The map 
\[\phi\colon  G\to \Alb(X), \quad g\mapsto \alpha(g\cdot x_0)\]
is a holomorphic covering map of complex Lie groups. In particular no holomorphic vector field on $X$ has zeros and the stabiliser of any point in $X$ is discrete in $G$.
\item The Albanese map is a surjective holomorphic submersion with connected fibres.
\item  Every fibre is a compact complex manifold with trivial canonical bundle.
\end{enumerate}
\end{thm}

\begin{rem}
If in Theorem \ref{thm: Albanese} $X$ is in Fujiki's class $\kc$, then the fibre of the Albanese map is also in class $\kc$, and hence satisfies the $\del\delbar$-Lemma.

Does this hold true if the total space $X$ is not in class $\kc$ but only satisfies the $\del\delbar$-Lemma? 

If the answer to this question is positive, then the fibre $F$ of the Albanese map would also satify the assumptions of Theorem \ref{thm: Albanese} and, therefore, we could inductively understand the structure of $X$.
\end{rem}

\begin{proof}[Proof of Theorem~\ref{thm: Albanese}]
It is well known that $G$ is a complex Lie group \cite{Bochner-Montgomery47}.
The first item follows immediately from Proposition~\ref{prop: contraction non-degenerate}, because the differential of the group homomorphism $\phi$,
\begin{equation}\label{eq: differential}
D_{\id}\phi \colon \gothg = (H^0(X, \kt_X), [ -, - ] ) \to H^0(X, \Omega_X^1)^*,
\end{equation}
is the map induced from the evaluation pairing $b$.

Since the orbit $G\cdot x_0$ maps surjectively onto $\Alb(X)$, so does $X$ and $\alpha$ is surjective. 

As changing the base point changes the Albanese map only by a translation in $\Alb(X)$, \eqref{eq: differential} also implies that $\alpha$ is a submersion in every point and hence every fibre $F$ is smooth with trivial canonical bundle $K_F = {K_X}_{|_F}$.

It remains to prove that $\alpha$ has connected fibres. For this consider the Stein factorisation 
\[\begin{tikzcd}
 X \rar{\text{conn.}}[swap]{\text{fibres}} \arrow{dr}[swap]{\alpha} & Y\dar{\text{finite}}[swap]{\beta}\\ & \Alb(X).
\end{tikzcd}\]
Because $\alpha$ is a submersion, the finite map $\beta$ is a submersion, that is, $Y$ is a complex torus as well and $\beta$ is an isogeny. By the universal property of the Albanese map, $\deg \beta = 1$ and $\alpha$ has connected fibres.
\end{proof}

\begin{rem}\label{rem: albanese Kaehler}
If in Theorem~\ref{thm: Albanese} the manifold $X$ is K\"ahler, then, for example by the Beauville--Bogomolov decomposition theorem \cite{bea}, the orbits of $G$ are compact and the map $G\to \Alb(X)$ is an isogeny of complex tori, so that after a finite cover $X$ splits as a product of a complex torus and a simply connected manifold.
 
The example of the deformation of the Nakamura manifold shows \cite{CattaneoTomassini16} that this conclusion does in general not hold without the K\"ahler assumption (cf.\ also Example~\ref{ex: nakamura} and \ref{ex: open orbit}).
\end{rem}

Theorem \ref{thm: Albanese} immediately implies:

\begin{cor}
Let $X$ be a compact complex manifold with trivial canonical bundle and satisfying the $\del\delbar$-Lemma. Then the Albanese map $\alpha\colon X\to \Alb(X)$ induces an injection 
\[\alpha^*\colon H^*(\Alb(X), \IC) \isom \Wedge^*H^1(X, \IC) \into H^*(X, \IC).\]
\end{cor}

We deduce as in \cite[Thm.\ 21]{Kawamata81}:

\begin{cor}\label{cor: q(X) leq dim X}
Let $X$ be a compact complex manifold with trivial canonical bundle and satisfying the $\del\delbar$-Lemma. Then $b_1(X) = 2 h^0(X, \Omega_X^1)\leq 2\dim X$ and equality holds if and only if $X$ is a complex torus.  
\end{cor}
\begin{proof}
By Theorem~\ref{thm: Albanese}, the Albanese map is surjective; hence, the inequality follows. In case of equality, $\alpha$ is a finite holomorphic submersion of degree $1$, i.e., an isomorphism. 
\end{proof}

\begin{cor}
If $b_1(X)>0$, then the topological Euler-characteristic vanishes.
\end{cor}
\begin{proof}
Differentiably, $X\to \Alb(X)$ is a locally trivial fibre bundle with typical fibre $F$ by Theorem~\ref{thm: Albanese}. Since the topological Euler number is multiplicative in fibre bundles, the result follows as soon as $\frac 12 b_1(X)  =  \dim \Alb(X)>0$.
\end{proof}
 
\begin{exam}[{cf.\ \cite[Ex.\ 4.2]{CattaneoTomassini16}}]\label{ex: nakamura}
Let $N = (\Gamma' \ltimes \Gamma'') \backslash (\IC \ltimes \IC^2)$ be the complex parallelisable Nakamura manifold and $T$ a $1$-dimensional complex torus. Let $X = N \times T$ and consider the deformation of $X$ defined by
\[\begin{array}{llcll}
(1, 0)\text{-forms:} & \varphi^1_t = dz^1 - t d\bar{z}^1 & \quad & (0, 1)\text{-forms:} & \omega^1_t = d\bar{z}^1 - \bar{t} dz^1\\
 & \varphi^2_t = e^{-z^1} dz^2 & \quad & & \omega^2_t = e^{-z^1} d\bar{z}^2\\
 & \varphi^3_t = e^{z^1} dz^3 & \quad & & \omega^3_t = e^{z^1} d\bar{z}^3\\
 & \varphi^4_t = dz^4 & \quad & & \omega^4_t = d\bar{z}^4.
\end{array}\]
It is then easy to see that the deformed manifold $X_t$ has the structure equations
\[\begin{array}{lcl}
d\varphi^1_t = 0 & \quad & d\omega^1_t = 0\\
d\varphi^2_t = -\frac{1}{1 - |t|^2} \varphi^1_t \wedge \varphi^2_t + \frac{t}{1 - |t|^2} \varphi^2_t \wedge \omega^1_t & \quad & d\omega^2_t = -\frac{1}{1 - |t|^2} \varphi^1_t \wedge \omega^2_t - \frac{t}{1 - |t|^2} \omega^1_t \wedge \omega^2_t\\
d\varphi^3_t = \frac{1}{1 - |t|^2} \varphi^1_t \wedge \varphi^3_t - \frac{t}{1 - |t|^2} \varphi^3_t \wedge \omega^1_t & \quad & d\omega^3_t = \frac{1}{1 - |t|^2} \varphi^1_t \wedge \omega^3_t + \frac{t}{1 - |t|^2} \omega^1_t \wedge \omega^3_t\\
d\varphi^4_t = 0 & \quad & d\omega^4_t = 0.
\end{array}\]
and that $X_t$ is a complex symplectic manifold by means of the form
\[\sigma_t = \varphi^{14}_t + \varphi^{23}_t.\]
We point out here that $X_t$ satisfies the $\del \delbar$-Lemma for every $t \neq 0$, while the central fibre does not by \cite{angella-kasuya17}.
Moreover, we have
\[H^0(X_t, \Omega^1_{X_t}) = \left\{ \begin{array}{ll}
\langle \varphi^1_0, \varphi^2_0, \varphi^3_0, \varphi^4_0 \rangle & \text{for } t = 0,\\
\langle \varphi^1_t, \varphi^4_t \rangle & \text{for } t \neq 0.
\end{array} \right.\]
We consider then the holomorphic map
\[\begin{array}{rccc}
f: & N_t \times T & \longrightarrow & (\Gamma \backslash \IC) \times T\\
   & ([(z^1_t, z^2_t, z^3_t)], [z^4_t]) & \longmapsto & ([z^1_t], [z^4_t]),
\end{array}\]
and so after the choice of a base point on $X_t$ we get a commutative diagram
\[\begin{tikzcd}
X_t \arrow{r}{f} \arrow{dr}[swap]{\alpha_t} & (\Gamma \backslash \IC) \times T\\
 & \Alb(X_t) \arrow{u}[swap]{g_t}.
\end{tikzcd}\]
We observe now that for $t \neq 0$, both $(\Gamma \backslash \IC) \times T$ and $\Alb(X_t)$ are $2$-dimensional complex tori, so $g_t$ is an isogeny and $\alpha_t^{-1}(p) = f^{-1}(g_t(p))$ for every $p \in \Alb(X_t)$. The fibres of $f$ are easy to describe: from the structure equations above we can easily see that they are all $2$-dimensional complex tori.

We want also to observe that the central fibre has trivial canonical bundle, does not satisfy the $\del \delbar$-Lemma and has $2h^{1,0}(X_0) =\dim X =  4$, but $X$ is not a complex torus. In fact, we have $H^0(X, d\ko_X) = \langle \phi_0^1,  \phi_0^4\rangle$ and the Albanese torus is only of dimension $2$.  So this example shows that Corollary~\ref{cor: q(X) leq dim X} is false without assuming the $\del \delbar$-Lemma. 
\end{exam}

\begin{exam}\label{ex: open orbit}

Here we consider again small deformations $N_t$ $(t\neq 0)$ of the Nakamura threefold to see Theorem \ref{thm: Albanese} in action and point out some differences to the K\"ahler case.

Following \cite[p.\ 98, case 1]{nakamura75} we can describe $N_t$ as the quotient of a solvable real Lie group $\tilde N$, which as a complex manifold is $\tilde N \isom \IC^3 = \IC \times \IC^2$, by the following action of a lattice $\Gamma = \Gamma' \times \Gamma''$, where $\Gamma' = \IZ \oplus \IZ \cdot 2\pi i$:
\[\begin{array}{rccc}
p_t(\omega_1, \omega_2, \omega_3): & \IC^3 & \longrightarrow & \IC^3\\
 & (z^1, z^2, z^3) & \longmapsto & (z^1 + \omega_1 + t \bar{\omega}_1, e^{\omega_1}z^2 + \omega_2, e^{-\omega_1}z^3 + \omega_3).
\end{array}\]
For small values of $t$, the map $\gamma_t: \Gamma' \longrightarrow \IC$ defined as $\gamma_t(\omega_1) = \omega_1 + t \bar{\omega}_1$ is injective, and its image is then a lattice isomorphic to $\Gamma'$.

Consider the action on $\tilde N$ given by 
\[ H\isom \IC \into \Aut_\ko(\tilde N), \quad c \mapsto \tau_c\]
where $\tau_c([(z^1, z^2, z^3)]) = [z^1 + c, z^2, z^3]$. 
Note that the action of $H$ descends to a holomorphic action on $N_t$, and since $\dim \Aut_\ko(N_t) = h^0(X, \kt_X) = h^{1,0}(X)=1$ by Proposition \ref{prop: contraction non-degenerate} we see that $H$ is the universal cover of the identity component of the holomorphic automorphism group of $N_t$. 

Let $x_0\in N_t$ be the image of the identity elemente $0\in \tilde N$. Then clearly the orbit $H\cdot 0\subset \tilde N$ is isomorphic to $\IC$ and it is easy to check that $H\cdot 0\cap \Gamma\isom \IZ$. Therefore the orbit $\Aut_\ko(N_t)_0 \cdot x_0$ is isomorphic to $\IC^*$ and not closed in $N_t$, because $N_t$ is compact.

In particular, the action of $\Aut_\ko(N_t)$ induces a holomorphic foliation without closed leaves on $N_t$ and the map $\Aut_\ko(N_t)\cdot x_0 \to \Alb(N_t)$ has infinite degree. This illustrates Theorem \ref{thm: Albanese} for the manifolds $N_t$, and also highlights the differences to the K\"ahler case, compare Remark \ref{rem: albanese Kaehler}.
\end{exam}
\begin{quest}
Can one extend, at least partially, the above results to a larger class of manifolds, assuming, for example, that $c_1(X)=0$ in $H^{1,1}_{BC}(X)$, advocated in \cite{Tosatti15}, or even only that the Kodaira dimension $\kappa(X)=0$?
\end{quest}

\section{Unobstructed deformations and stability of properties}\label{sect: unobstructed defs}
Eventually, we want to study deformations of $\del\delbar$-complex symplectic manifolds, but the first results go through under the weaker assumption that the $\del\delbar$-Lemma holds and the canonical bundle is trivial, so we work in this setting.

\subsection{Stability of properties}
To have a sensible theory we need to show that the property of being $\del\delbar$-symplectic is open in the universal family of deformations. Again, this also works in the more general setting where we just assume that the canonical bundle is trivial. In the hyperk\"ahler setting this was done in \cite[Prop.\ 9 (p.\ 771)]{bea}.

\begin{prop}\label{prop: stability-under-defs}
Let $X$ be a compact complex manifold and let $f\colon \kx\to B$ be a small deformation of $X = \kx_0$.
If $X$ satisfies the $\del\delbar$-Lemma (or if the Fr\"olicher spectral sequence for $X$ degenerates on the first page), then every sufficiently close neighbouring fibre satisfies the $\del\delbar$-Lemma (or its Fr\"olicher spectral sequence degenerates on the first page, respectively,) as well.
Assuming either of these,  
\begin{enumerate}
\item if $K_X$ is trivial, then $K_{X_t}$ is trivial for $t$ sufficiently close to $0$;
\item if $X$ admits a complex symplectic form, then $X_t$ admits a complex symplectic form for $t$ sufficiently close to $0$.
\end{enumerate}
\end{prop}
\begin{proof}
The deformation openness of the $\del\delbar$-Lemma or the $E_{1}$-degeneration of the Fr\"olicher spectral sequence are known; for the convenience of the reader we include a proof for the former, following \cite[Cor.\ 2.7]{AngellaTomassini}, the latter being similar but easier.

By loc.\ cit., on a compact complex manifold $Y$,
\[2 b_k(Y) \leq \sum_{p + q = k} h^{p, q}_{BC}(Y) + h^{p, q}_{A}(Y),\]
and equality holds for every $k$ if and only if $Y$ satisfies the $\del \delbar$-Lemma (here $h^{p, q}_{BC}(Y)$ and $h^{p, q}_A(Y)$ are the dimensions of Bott--Chern and Aeppli cohomologies, respectively).
Hence, the result follows by upper semi-continuity of these numbers: indeed, if $\mathcal{Y}_t$ is a small deformation of $Y = \mathcal{Y}_0$, then
\[2 b_k(Y) = \sum_{p + q = k} h^{p, q}_{BC}(\mathcal{Y}_0) + h^{p, q}_{A}(\mathcal{Y}_0) \geq \sum_{p + q = k} h^{p, q}_{BC}(\mathcal{Y}_t) + h^{p, q}_{A}(\mathcal{Y}_t) \geq 2 b_k(\mathcal{Y}_t) = 2 b_k(Y).\]

Then the Hodge numbers $h^{p, q}(\kx_s)$ are constant: by upper semi-continuity (\cite[Cor.\ 9.19]{voisin}) we have that $h^{p, q}(\kx_s) \leq h^{p, q}(\kx_0)$, and since
\[\sum_{p + q = k} h^{p, q}(\kx_s) = b_k(\kx_s) = b_k(\kx_0) = \sum_{p + q = k} h^{p, q}(\kx_0),\]
we deduce that $h^{p, q}(\kx_s) = h^{p, q}(\kx_0)$ for every $(p, q)$.
Thus, for all $k$ the sheaf $f_* \Omega^k_{\kx|B}$ is locally free with fibre  $H^0(\kx_t, \Omega^k_{\kx_t})$ at $t$.
Shrinking $B$ further, if necessary, we may assume that all these vector bundles are trivial, so that for every holomorphic form on the central fibre, we can choose an extension to the total space. 

To conclude the proof, it thus suffices to observe that the locus where a holomorphic $n$-form does not vanish, respectively, a holomorphic $2$-form is non-degenerate, is open and thus, by properness of the map, contains a neighbourhood of the central fibre, as claimed.
\end{proof}

\begin{rem}
The property of $K_X$ being trivial is not open in a family without assuming  $E_{1}$-degeneration of the Fr\"olicher spectral sequence, as the example in \cite{Ueno80} shows.
\end{rem}

\subsection{Unobstructed deformations}
It seems to be well known that the K\"ahler condition in the celebrated Tian--Todorov Theorem on unobstructedness of deformations of Calabi--Yau manifolds can be weakened, but it appears that in the literature this is at best stated implicitly.

For example, after minor changes, the proof in \cite[Ch.\ 6]{Huybrechts} works for manifolds satisfying the $\del\delbar$-Lemma.
Even stronger, the proof in \cite[p.\ 154]{KKP08} shows:

\begin{thm}[Generalised Tian--Todorov Theorem]\label{thm: unobstructed}
Let $X$ be a compact complex manifold with trivial canonical bundle.
If the Fr\"olicher spectral sequence for $X$ degenerates on the first page, then $X$ has unobstructed deformations.
\end{thm}
For convenience of the reader we present a rough tour through the proof by Kontsevich, Katzarkov and Pantev and provide some supplementary details where the degeneration assumption is used.
\begin{proof}[Sketch of proof.]
The following arguments show smoothness of the formal moduli space, which implies that the Kuranishi space is smooth as well; hence, deformations are unobstructed.

To show that the formal deformation space is smooth, it has to be shown that the Kodaira-Spencer dg Lie algebra is homotopy abelian (loc.\ cit.\ Definition 4.9).
This is achieved through the fact that it is a direct summand of another dg Lie algebra, so that it suffices to show that this larger one is homotopy abelian.
But by loc.\ cit.\ Theorem 4.14 (1), this holds as soon as the corresponding dg Batalin-Vilkovisky algebra satisfies the so-called degeneration property (loc.\ cit.\ Definition 4.13; cf.\ below).
This algebra happens to be isomorphic to the Dolbeault double complex $(\ka^{*,*}_{X},\delbar,\del)$.
To complete the proof, all that is left to show is that the latter satisfies the degeneration property; that is, we have to show that for each positive integer $N\geq 1$, the cohomology $H(\ka^{*,*}_{X}\otimes_{\IC}\IC[u]/(u^N),\delbar+u\del)$ is a free $\IC[u]/(u^N)$-module.

To this end, we compare the spectral sequences $(E^{\ast,\ast}_{r})'$ and $(E^{\ast,\ast}_{r})''$  associated with the double complexes $(\ka^{*,*}_{X}\otimes_{\IC}\IC[u]/(u^N),u\del,\delbar)$ and  $(\ka^{*,*}_{X}\otimes_{\IC}\IC[u]/(u^N),\del,\delbar)$, computing the cohomology of their total complexes $H(\ka^{*,*}_{X}\otimes_{\IC}\IC[u]/(u^N),\delbar+u\del)$ and $H(\ka^{*,*}_{X}\otimes_{\IC}\IC[u]/(u^N),d)$, respectively.
They have the same cohomology groups on the first page and $d_{1}' = ud_{1}''$.
Furthermore, $(E_{1})'' = E_{1}\otimes_{\IC}\IC[u]/(u^N)$ is obtained from the Fr\"olicher spectral sequence of $X$ by scalar extension.
Thus, if the Fr\"olicher spectral sequence degenerates at the first page, so do the other two; a posteriori, they even become isomorphic.
Consequently, $H(\ka^{*,*}_{X}\otimes_{\IC}\IC[u]/(u^N),\delbar+u\del)$ is isomorphic to $H(X,\IC)\otimes_{\IC}\IC[u]/(u^{N})$ through those spectral sequences and so it is free over $\IC[u]/(u^{N})$, as claimed.
\end{proof}

\begin{rem}
Without assumptions like the degeneration of the Fr\"olicher spectral sequence on the first page, the result is definitely far from true. As shown in \cite{rollenske11}, most complex parallelisable nilmanifolds have obstructed deformations, for example if they contain an abelian factor. 
\end{rem}

\begin{cor}\label{cor: universal deformation}
 Let $X$  be a compact complex manifold with trivial canonical bundle whose Fr\"olicher spectral sequence degenerates at the first page. Then the Kuranishi family of $X$ is a smooth universal deformation. 
\end{cor}
\begin{proof}
By Theorem~\ref{thm: unobstructed}, the Kuranishi family is indeed smooth.
To see that is also universal we have to check that the number of independent holomorphic vector fields remains constant in a neighbourhood of the central fibre by Wavrik's theorem \cite{Wavrik69}.
By Proposition~\ref{prop: stability-under-defs}, nearby fibres $X_t$ also have $E_{1}$-degeneration of the Fr\"olicher spectral sequence, hence, the Hodge numbers remain constant.
Thus, using the isomorphism of the tangent bundle and the bundle of holomorphic $n-1$ forms induced by a trivialisation of the canonical bundle we have, 
\[h^0(X_t, \kt_{X_t}) = h^0(X_t, \Omega^{n-1}_{X_t}) = h^0(X, \Omega^{n-1}_{X}) = h^0(X, \kt_{X}),\] 
which concludes the proof. 
\end{proof}

\section{Local Torelli for \texorpdfstring{$\del\delbar$}{del-delbar}-complex symplectic manifolds}\label{sec: local torelli}
We now study period maps of $\del \delbar$-complex symplectic manifolds, closely following Huybrechts' exposition \cite[22.3]{HuybrechtsHKM}.
Similar results were obtained (in a much more conceptual way) by Kirschner for singular symplectic spaces of Fujiki's class $\kc$ \cite{Kirschner:2015}.

\begin{thm}[Local Torelli]\label{thm: general local torelli}
Let $(X, \sigma)$ be a $\del\delbar$-complex symplectic manifold. Then the period map for the Hodge structure on $H^2(X, \IC)$, 
\[\kp_{X}\colon \Def(X)\to \Grass(h^{2,0}(X),H^{2}(X,\IC)),\quad s\mapsto [H^{2,0}(\kx_{s})],\]
is an immersion.
\end{thm}
\begin{proof}
Let $\Def(X)$ be the universal deformation space of $X$, which exists and is smooth by Corollary~\ref{cor: universal deformation}.
We only need to show that the differential of the period map,  
\[d\kp_X\colon T_0\Def(X) = H^1(X, \kt_X) \to 
\Hom\left(H^{2,0}(X), H^2(X, \IC)/H^{2,0}(X)\right), \]
is injective, that is, for any $\kappa\in H^1(X, \kt_X)$ the homomorphism $\kp_X(\kappa)$ is non-zero. Evaluating on the symplectic form $\sigma$ we have $d\kp_X(\kappa) (\sigma) = \kappa \lrcorner \sigma $, by Giffiths' description of the derivative, and this is non-zero, because  contraction with the symplectic form $\sigma$ induces the isomorphism $ H^1(X, \kt_X)\isom H^{1,1}(X)\subset H^2(X, \IC)$.
\end{proof}

\subsection{The period map and the Beauville--Bogomolov--Fujiki form}
One of the most useful features of the Beauville--Bogomolov--Fujiki quadratic form in hyperk\"ahler geometry  is its relation to the period map.
We will now show that this extends to the case of simple $\del\delbar$-complex symplectic manifolds, but fails in the general case.
\begin{defin}\label{def: bbf-form}
Let $(X,\sigma)$ be a $\del\delbar$-complex symplectic manifold of dimension $2n$.
The \emph{Beauville--Bogomolov--Fujiki quadratic form} $q_{\sigma}\colon H^{2}(X,\IC)\to \IC$ is defined as 
$$q_{\sigma}(\alpha) = \frac{n}{2}\int_{X}(\sigma\bar{\sigma})^{n}\int_{X}\alpha^{2}(\sigma\bar{\sigma})^{n-1}+(1-n)\int_{X}\alpha\sigma^{n-1}\bar{\sigma}^{n}\int_{X}\alpha\sigma^{n}\bar{\sigma}^{n-1}.$$
\end{defin}
Quite often, it is assumed that the symplectic form is normalised in such a way that $\int_{X}(\sigma\bar{\sigma})^{n} = 1$.
This is not a restriction since rescaling by a complex scalar $t\in\IC^\times$ has the effect that $q_{t\sigma} = |t|^{4n-2}q_{\sigma}$.
(We learned the correct version of $q_{\sigma}$ for non-normalised forms from \cite{Lehn:05}.)
In particular, the quadric in $\IP(H^2(X,\IC))$ defined by $q_{\sigma}$ is independent of the normalisation.  
For a proof that this indeed defines a quadratic form and for its most important properties we refer to \cite{CattaneoTomassini16}.

We will show that for simple $\del\delbar$-complex symplectic manifolds, the image of the period map is contained in the quadric defined by the Beauville--Bogomolov--Fujiki quadratic form.
For sake of generality, we first formulate a result which still holds in the general case (cf.\ Remark~\ref{rem: period map}).

\begin{lem}\label{lem: bbf-vanishing}
Let $(X,\sigma)$ be a $2n$-dimensional $\del\delbar$-complex symplectic manifold.
For each element $\alpha\in H^2(X,\IC)$ which decomposes as $\alpha = \lambda\sigma+\alpha_{(1,1)}+\mu\bar{\sigma}$, where $\lambda,\mu\in\IC$ and $\alpha_{(1,1)}\in H^{1,1}(X)$, we have $\int_{X}(\sigma\bar{\sigma})^{n}\int_{X}\alpha^{n+1}\bar{\sigma}^{n-1} = (n+1)\lambda^{n-1}q_{\sigma}(\alpha).$
In particular, if $\lambda \not= 0$ and $\alpha^{n+1}=0$, then $q_{\sigma}(\alpha) = 0$.
\end{lem}
Note that if $X$ is simple, then every element of $H^2(X,\IC)$ has such a decomposition, so that this becomes an empty condition in this case.
\begin{proof}(Cf.\ \cite[proof of Lemma 22.9]{HuybrechtsHKM})
For $\alpha$ of this particular form, it is easy to compute
$$q_{\sigma}(\alpha) = \lambda\mu\left(\int_{X}(\sigma\bar{\sigma})^n\right)^2+\frac{n}{2}\int_{X}\alpha_{(1,1)}(\sigma\bar{\sigma})^{n-1}\int_{X}(\sigma\bar{\sigma})^n$$
and
$$\int_{X}\alpha^{n+1}\bar{\sigma}^{n-1} = (n+1)\lambda^n\mu\int_{X}(\sigma\bar{\sigma})^n+\binom{n+1}{2}\lambda^{n-1}\int_{X}\alpha_{(1,1)}(\sigma\bar{\sigma})^{n-1}.$$
These readily yield the claimed identity $\int_{X}(\sigma\bar{\sigma})^{n}\int_{X}\alpha^{n+1}\bar{\sigma}^{n-1} = (n+1)\lambda^{n-1}q_{\sigma}(\alpha)$.
The \emph{in particular}-part is clear and the proof is complete.
\end{proof}

\begin{cor}\label{cor: simple bbf vanishing}
Let $(\kx,\Sigma)\to S$ be a deformation of a simple $\del\delbar$-complex symplectic manifold $(X,\sigma) = (\kx_{0},\Sigma|_{\kx_{0}})$ and denote $\sigma_{s} := \Sigma|_{\kx_{s}}\in H^{2,0}(\kx_{s})$ for each $s\in S$.
Then there exists an open neighbourhood $U\subset S$ of $0$ such that $q_{\sigma}(\sigma_{s}) = 0$ and $q_{\sigma}(\sigma_{s}+\bar{\sigma_{s}})>0$ for each $s\in U$.
\end{cor}
\begin{proof}
With respect to the complex structure for $\kx_{s}$, the class $\sigma_{s}$ is of type $(2,0)$; thus, its powers beyond $n := \tfrac{1}{2}\dim(X)$ vanish.
Furthermore, in the type decomposition $\sigma_{s} = \lambda_{s}\sigma+(\sigma_{s})_{(1,1)}+\mu_{s}\bar{\sigma}$ (with respect to $X=\kx_{0}$), the coefficient $\lambda_{s}$ is different from zero for $s$ sufficiently close to $0$, for continuity reasons.
Thus, Lemma~\ref{lem: bbf-vanishing} gives $q_{\sigma}(\sigma_{s}) = 0$ in an open neighbourhood of $0$.
Similarly, $q_{\sigma}(\sigma_{s}+\bar{\sigma_{s}})$ is real and varies continuously with $s\in S$; hence, $q_{\sigma}(\sigma+\bar{\sigma}) = (\int_{X}(\sigma\bar{\sigma})^{n})^2>0$ implies that $q_{\sigma}(\sigma_{s}+\bar{\sigma_{s}})>0$ for each $s$ in a certain open neighbourhood, as claimed.
\end{proof}

If $(X,\sigma)$ is a simple $\del\delbar$-complex symplectic manifold, then so are the neighbouring fibres in the universal deformation $\kx\to \mathrm{Def}(X)$ (cf.\ Corollary~\ref{cor: universal deformation}) and we can choose a $\Sigma$ extending $\sigma\in H^{2,0}(X)$, and in fact there is only one such up to invertible resacling.
In particular, the line spanned by $\Sigma|_{\kx_{s}}$ in $H^2(X,\IC)$ does not depend on the choice of $\Sigma$; clearly, this recovers the period map.
Therefore, we conclude:

\begin{cor}
Let $(X,\sigma)$ be a simple $\del\delbar$-complex symplectic manifold.
Then the period map $\kp_{X}\colon \mathrm{Def}(X)\to \IP(H^{2}(X,\IC)^{\vee})$ takes values in
$$Q_{X} := \{\IC\alpha\in  \IP(H^{2}(X,\IC)^{\vee})\mid q_{\sigma}(\alpha)=0\text{ and } q_{\sigma}(\alpha+\bar{\alpha})>0\}.$$
\end{cor}

With this result at hand, we can finally state the strengthened local Torelli theorem in the simple case.

\begin{thm}[Local Torelli Theorem, simple case]\label{thm: simple local torelli}
For a simple $\del\delbar$-complex symplectic manifold $X$, the period map $\kp\colon\Def(X)\to Q_{X}$ is a local isomorphism.
\end{thm}
\begin{proof}
The spaces $\Def(X)$ and $Q_{X}$ are smooth (by Theorem~\ref{thm: unobstructed}, respectively, \cite[Theorem 2]{CattaneoTomassini16}), and have the same dimension $h^{1,1}(X) = h^{2}(X)-2$.
Hence, the claim follows from the general local Torelli Theorem \ref{thm: general local torelli}.
\end{proof}

For later reference we record the following consequence.
\begin{cor}\label{cor: quadrics agree on simple}
If $(X,\sigma)$ is a simple $\del\delbar$-complex symplectic manifold and if $\sigma'$ is the symplectic structure on a nearby fibre in the universal deformation space, then $q_{\sigma}$ and $q_{\sigma'}$ define the same quadric in $H^2(X,\IC)$.
\end{cor}
\begin{proof}
In fact, the universal deformation space of $X$ is also the universal deformation space for all nearby fibres $\kx_{s}$ and the implicit identification $H^2(\kx_{s},\IC)\cong H^2(X,\IC)$ granted by Ehresmann's Theorem is compatible with the respective period maps.
Thus, the two smooth hypersurfaces in $\IP(H^2(X,\IC))$ defined by $q_{\sigma}$ and $q_{\sigma_{s}}$ have the image of the peroid map as an open subset in common; consequently, they agree.
\end{proof}

We give a sample application how this is often used in the theory of hyperk\"ahler manifolds.
\begin{prop}
Let $(X, \sigma)$ be a simple $\del\delbar$-complex symplectic manifold. Then a very general small deformation of $X$ has algebraic dimension zero and does not contain any effective divisor. 
\end{prop}
\begin{proof}
Consider the countable union of hyperplanes of the form $\alpha^\perp\subset \IP (H^2(X, \IC))$ for all $\alpha\in H^2(X, \IQ)\setminus 0$, where the orthogonal complement is computed with respect to the quadratic form $q_\sigma$.
The complement of this union in $Q_{X}$, say $V\subset Q_{X}$, is inhabited, as $h^{1,1}(X)>0$ (by \cite[Theorem 1 \& 2]{CattaneoTomassini16}).
Since the period map is a local isomorphism onto $Q_X$, we can, therefore, choose a small deformation $(X',\sigma')$ whose period point $[\sigma']$ lies in $V$.
By Corollary~\ref{cor: quadrics agree on simple}, the quadratic forms $q_{\sigma}$ and $q_{\sigma'}$ agree up to an invertible scalar factor and so $\sigma'$ is not orthogonal to any $\alpha\in H^2(X,\IQ)\setminus 0$ also with respect to $q_{\sigma'}$.
But by \cite[Lemma 2.11]{CattaneoTomassini16}, every $(1,1)$-class is orthogonal to $\sigma'$; hence, $H^{1,1}_{\delbar}(X')\cap  H^2(X, \IQ) =0$.

Now assume this small deformation $(X', \sigma')$ contains an effective divisor $D$.
Then the class of $D$ in cohomology is a rational class of type $(1,1)$ and, hence, trivial.
By the $\del\delbar$-Lemma, we have $[i\del\delbar f] = c_1(D)$ for some smooth function $f$, which is plu\-ri\-sub\-har\-mon\-ic since $i\del\delbar f$ is the current of integration along $D$, hence positive by Lelong's theorem (cf. \cite[IV 3 Ex.\ 3.2]{BHPV} and the references therein, for example).
But plurisubharmonic functions on compact manifolds are constant and so $D$ is trivial.

In particular, all meromorphic functions on $X'$ are constant, for any non-constant meromorphic function would give rise to a divisor.
\end{proof}

\subsection{Period maps in case \texorpdfstring{$h^{2,0}>1$}{h20 > 1}}\label{sect: non-primitive periods}
For a non-simple  $\del\delbar$-complex symplectic manifold $(X, \sigma)$, the period map considered in Section~\ref{sec: local torelli} maps into the Gra{\ss}mannian variety $\Grass(h^{2,0}(X),H^{2}(X,\IC))$, whereas the quadratic form $q_\sigma$ defines a quadric in $\IP(H^2(X,\IC))$.
However, if we consider a family of complex symplectic manifolds $f\colon\kx\to S$ together with a family of symplectic forms $\sigma_{s}\in H^{2,0}(\kx_{s})$, depending holomorphically on $s\in S$, we can consider the map $S\to \IP(H^2(X,\IC))$, $s\mapsto \IC\sigma_{s}$, resembling the period map in the simple case.

A natural question arises here, namely, whether the image of this map is contained in the quadric defined by $q_{\sigma_{0}}$, at least for $s\in S$ sufficiently close to $0\in S$.
The examples provided below (Example~\ref{ex: 4-torus} and \ref{ex: products}) show that this is not the case.

Nonetheless, it seems worthwhile to make the above idea precise in a universal fashion.
Let $f\colon \kx\to\Def(X)$ be the universal family.
We consider the pullback $\kp^{*}\gothU$ of the tautological vector bundle $\gothU\subset H^2(X,\IC)\times \Grass(h^{2,0}(X),H^2(X,\IC))$ on the Gra{\ss}mannian variety along the period map $\kp\colon \Def(X)\to \Grass(h^{2,0}(X),H^2(X,\IC))$.
For sake of properness, we pass to the projectivisations $\IP(\kp^{*}\gothU)\to \IP(\gothU)$.
Over a point $s\in\Def(X)$, this gives the linear inclusion $\IP(H^{2,0}(\kx_{s}))\subset \IP(H^{2}(X,\IC))$.
The subset of $\IP(\kp^{*}\gothU)$ consisting of the classes of symplectic forms is open.
Therefore, the germ $\Def(X,\sigma)\subset \IP(\kp^{*}\gothU)$ of the class of a symplectic form $\sigma \in H^{2,0}(X)$ is an analytic germ with a natural map $\Def(X,\sigma)\to\Def(X)$, whose fibre over $s\in \Def(X)$ consists of the classes $[\sigma_{s}]\in H^{2,0}(\kx_{s})$ of symplectic structures on $\kx_{s}$ near $\sigma_{0}$.
It seems natural to consider the map into the partial flag manifold
\begin{align}\label{eq: refined period map (to flag)}
\Def(X,\sigma)\to\mathrm{Flag}(1,h^{2,0}(X);H^2(X,\IC)),\quad (\kx_{s},[\sigma_{s}])\mapsto([\sigma_{s}],[H^{2,0}(\kx_{s})]),
\end{align}
refining the period map.
Note that the composition of this map with the projection onto the Gra{\ss}mannian $\Grass(h^{2,0},H^2(X,\IC))$ recovers the period map; since the latter is injective by the local Torelli theorem~\ref{thm: general local torelli}, so is this period-like map.

The composition of the map \eqref{eq: refined period map (to flag)} with the projection of the flag variety onto $\IP(H^2(X,\IC))$ gives the map
\begin{align}\label{eq: naive period map}
\Def(X,\sigma)\to \IP(H^2(X,\IC)),\quad (\kx_{s},[\sigma_{s}])\mapsto[\sigma_{s}],
\end{align}
resembling the period map in the simple case even if $X$ is not simple.
We will refer to it as the \emph{naive period map}.

The following  examples show that for non-simple $\del \delbar$-complex symplectic manifolds the image of the naive period map \eqref{eq: naive period map} may not be contained in the quadric defined by the Beauville--Bogomolov--Fujiki form.

\begin{exam}[Complex tori]\label{ex: 4-torus}
Let $X$ be a $4$-dimensional complex torus and denote as usual by $dz^1, \ldots, dz^4$ the standard holomorphic coframe of $(1, 0)$-forms. Consider the complex deformation of $X$ given by
\[\left\{ \begin{array}{l}
dw^1 = dz^1 + t_1 d\bar{z}^3\\
dw^2 = dz^2 + t_2 d\bar{z}^4\\
dw^3 = dz^3 + t_1 d\bar{z}^1\\
dw^4 = dz^4 + t_2 d\bar{z}^2,
\end{array} \right.\]
and the complex symplectic form on $X_t$
\[\sigma_t = dw^{12} + dw^{34} + t_3 dw^{13} + t_4 dw^{24},\]
where $t = (t_1, t_2, t_3, t_4)$ varies in a small polydisc centered in the origin of $\IC^4$. Observe that on the central fibre $X = X_0$ the form $\sigma_0$ reduces to the standard symplectic form $\sigma = dz^{12} + dz^{34}$. We will show that $q_\sigma(\sigma_t) \neq 0$, namely, that
\[q_{\sigma}(\sigma_t) = \int_{X} (\sigma\bar{\sigma})^{2} \int_{X} \sigma_t^{2} \sigma \bar{\sigma} - \int_{X} \sigma_t \sigma \bar{\sigma}^{2} \int_{X} \sigma_t \sigma^{2} \bar{\sigma} \neq 0.\]
Set $\Vol = dz^{1234} \wedge dz^{\bar{1}\bar{2}\bar{3}\bar{4}}$, then we can compute the four integrals involved:
\begin{enumerate}
\item $\int_{X} (\sigma\bar{\sigma})^{2} = 4 \int_X \Vol$;
\item $\int_{X} \sigma_t^{2} \sigma \bar{\sigma} = 4 t_1 t_2 (1 - t_3 t_4) \int_X \Vol$;
\item $\int_{X} \sigma_t \sigma \bar{\sigma}^{2} = 4 \int_X \Vol$;
\item $\int_{X} \sigma_t \sigma^{2} \bar{\sigma} = 4 t_1 t_2 \int_X \Vol$.
\end{enumerate}
Therefore,
\[q_{\sigma}(\sigma_t) = -16 t_1 t_2 t_3 t_4 \left( \int_X \Vol \right)^2,\]
which is clearly not identically zero on the polydisc.
\end{exam}

\begin{exam}[Products of $\del\delbar$-complex symplectic manifolds]\label{ex: products}
We now consider the deformation space and period map for products of $\del\delbar$-complex symplectic manifolds. For simplicity, we stick to the case of two factors. 

So assume that $(X_1, \sigma_1)$ and $(X_2, \sigma_2)$ are $\del\delbar$-complex symplectic manifolds and let $X=X_1\times X_2$. Assume further that $H^1(X_2, \IC)=0$.
Then we have 
\begin{equation}\label{eq: product of H2}
H^2(X, \IC) = H^2(X_1, \IC) \oplus H^2(X_2, \IC) \supset H^{2,0}(X) = H^{2,0}(X_1)\oplus H^{2,0}(X_2)
\end{equation}
and 
\[H^1(X, \kt_X) = H^1(X_1, \kt_{X_1})\oplus H^1(X_2, \kt_{X_2}),\]
where the latter is proved either using the isomorphism $\kt_X\isom \Omega^1_X$ provided by the symplectic form or simply by the Leray spectral sequence for the projection onto one factor. Note that this behaviour hinges on $b_1(X_2)=0$, as products of tori show.

Since $\Def(X)$ is universal by Corollary~\ref{cor: universal deformation} we conclude that 
\[\Def(X) = \Def(X_1)\times \Def(X_2).\]

Therefore, the period map $\kp_X$ can be decomposed as in the following diagram:
{\small
\[\begin{tikzcd}
\Def(X) \arrow[equal]{d}\arrow{r}{\kp} & \Grass(h^{2,0}(X), H^2(X, \IC))\\
\Def({X_1})\times \Def({X_2}) \rar{\kp_1\times \kp_2} & \Grass(h^{2,0}(X_1), H^2(X_1, \IC))\times \Grass(h^{2,0}(X_2), H^2(X_2, \IC)) \arrow{u}[swap]{(U_1, U_2)\mapsto U_1\oplus U_2}
\end{tikzcd}\]
}
We see that even if we start with $X_1$ and $X_2$ simple, the codimension of the image of the period map increases drastically.

We now specialise this example to show that the image of the naive period map \eqref{eq: naive period map} is not contained in the zero-locus of the Beauville--Bogomolov--Fujiki quadratic form.

Let $X_1$ and $X_2$ be two $K3$ surfaces, and $X = X_1 \times X_2$. Denote by $\sigma_i \in H^{2, 0}(X_i)$ generators such that $\int_{X_i}\sigma_i\bar\sigma_i = 1$. For ease of notation we do not distinguish forms on $X_i$ and their pullbacks, that is, the product symplectic form on $X$ is $\sigma = \sigma_1+\sigma_2$ and it satisfies  $\int_{X} (\sigma \bar{\sigma})^2 = 4$ by our normalisation for the $\sigma_i$ chosen above. 

Observe that on $X_i$ the Beauville--Bogomolov--Fujiki form associated to $\sigma_i$ reads as
\[q_{\sigma_i}(\varphi_i) = \frac{1}{2} \int_{X_i} \varphi_i^2, \qquad \varphi_i \in H^2(X_i, \IC),\]
and in particular it is independent of $\sigma_i$.  The Beauville--Bogomolov--Fujiki quadratic form associated to $\sigma$ is then
\[q_\sigma(\varphi) = 4 \int_X \varphi^2 \sigma \bar{\sigma} - \int_X \varphi \sigma \bar{\sigma}^2 \int_X \varphi \sigma^2 \bar{\sigma},\]
and decomposing $\varphi =  \varphi_1 + \varphi_2$, where $\phi_i\in H^2(X_i, \IC)$,  we can compute that
\begin{align*}
\int_X \varphi^2 \sigma \bar{\sigma} &= \int_{X_1} \varphi_1^2 + 2 \int_{X_1} \varphi_1 \sigma_1 \int_{X_2} \varphi_2 \bar{\sigma}_2 + 2 \int_{X_1} \varphi_1 \bar{\sigma}_1 \int_{X_2} \varphi_2 \sigma_2 + \int_{X_2} \varphi_2^2;\\
\int_X \varphi \sigma \bar{\sigma}^2 &= 2 \int_{X_1} \varphi_1 \bar{\sigma}_1 + 2 \int_{X_2} \varphi_2 \bar{\sigma}_2;\\
\int_X \varphi \sigma^2 \bar{\sigma} &= 2 \int_{X_1} \varphi_1 \sigma_1 + 2 \int_{X_2} \varphi_2 \sigma_2. 
\end{align*}

Hence,
\[q_\sigma(\varphi) = 8(q_{\sigma_1}(\varphi_1) + q_{\sigma_2}(\varphi_2)) - 4 \left( \int_{X_1} \varphi_1 \bar{\sigma}_1 - \int_{X_2} \varphi_2 \bar{\sigma}_2 \right) \left( \int_{X_1} \varphi_1 \sigma_1 - \int_{X_2} \varphi_2 \sigma_2 \right).\]
It is then possible to find a deformation $(X_{1, t}, \sigma_{1, t})$ such that the projection of $\sigma_{1, t}$ on the $\sigma_1$-axis is close to (but different from) $\sigma_1$ and the projection on the $\bar{\sigma}_1$-axis is close to $0$. Consider then the induced deformation $(X_t, \sigma_t) = (X_{1, t} \times X_2, \sigma_{1, t} + \sigma_2)$ of $(X, \sigma)$: to compute $q_\sigma(\sigma_t)$ we observe that
\begin{enumerate}
\item $q_{\sigma_1}(\sigma_{1, t}) = q_{\sigma_2}(\sigma_2) = 0$ since a $K3$ surface is simple;
\item $\int_{X_1} \sigma_{1, t} \bar{\sigma}_1 - \int_{X_2} \sigma_2 \bar{\sigma}_2$ is close to zero, and different to zero for $t \neq 0$;
\item $\int_{X_1} \sigma_{1, t} \sigma_1 - \int_{X_2} \sigma_2 \sigma_2 = \int_{X_1} \sigma_{1, t} \sigma_1$ is close to $1$.
\end{enumerate}
So this means that $q_\sigma(\sigma_t) \neq 0$ for $t \neq 0$.

Here is a more concrete example. We take $X_1$ to be the Kummer surface associated to a $2$-dimensional torus, and consider the deformation of $X_1$ induced by a deformation of the torus. In particular, if we let $dz^1, dz^2$ be a basis for the $(1, 0)$-forms on the torus, then we can consider
\[\left\{ \begin{array}{l}
dw^1 = dz^1 + t d\bar{z}^2\\
dw^2 = dz^2 + t d\bar{z}^1,
\end{array} \right.\]
and $\sigma_t$ induced on $X_{1, t}$ by the invariant form
\[(1 + t) dw^{12} = (1 + t) (dz^{12} + t dz^{1 \bar{1}} - t dz^{2 \bar{2}} - t^2 dz^{\bar{1} \bar{2}})\]
on the deformed torus.

Thus, the image of the naive period map is not contained in the zero locus of $q_\sigma$.
\end{exam}

\begin{rem}\label{rem: period map}
There are two more observations concerning the image of the naive period map \eqref{eq: naive period map} that should be mentioned.
If $(\kx,\Sigma)\to S$ is a deformation of a $\del\delbar$-complex symplectic manifold $(X,\sigma) = (\kx_{0},\Sigma|_{\kx_{0}})$ which does not change the complex structure, i.e., only varies the symplectic form, then $q_{\sigma}(\sigma_{s}) = 0$ for all $s\in S$, simply because $q_{\sigma}$ is trivial on $H^{2,0}(X)$.
In particular, the image of the naive period map \eqref{eq: naive period map} is contained in the quadric defined by $q_{\sigma}$ unless $X$ varies non-trivially.

Likewise, the conclusion of Corollary~\ref{cor: simple bbf vanishing} holds also in the non-simple case as long as the deformed symplectic form remains in $\mathop{Span}\{\sigma,\overline{\sigma}\}\oplus H^{1,1}(\kx_{0})\subset H^2(X,\IC)$, by the same proof.
Those deformations will be controlled by the quadric defined by $q_{\sigma}$ in $\IP(\mathop{Span}\{\sigma,\overline{\sigma}\}\oplus H^{1,1}(\kx_{0}))\subset \IP(H^2(X,\IC))$, again by the same line of arguments.
However, those special deformations seem to be of little interest. 
\end{rem}

\subsection{Period map for complex tori}\label{sect: period map tori}
We discuss one further case where the period map can be described explicitly.

Consider the real torus $T^{2n} = \IR^{2n}/\IZ^{2n}$ and let $V = \IR^{2n}\tensor\IC$. A complex structure $J_U$  on $T^{2n}$ is given by a decomposition $V^* = U\oplus \bar U$, where $U$ is identified with the space of $(1,0)$-forms. In other words, if $\pi\colon \IR^{2n}\into V \onto U^*$ is the natural projection, then $X_U = (T^{2n}, J_U)  = U/\pi(\IZ^{2n})$. 

It is well known that 
\[H^k(X_U, \IC) =\Wedge^k V^* =\bigoplus_{p+q=k} \Wedge^p U\tensor \Wedge^q\bar U\] 
is the Hodge decomposition of $X_U$ and that the local period map for the Hodge structure on $H^1$, 
\[\kp^1\colon\Def_{X_U} \to  \Grass (n, V^*),\]
is an  immersion onto an open subset, the Siegel upper half space. 
Thus, the period map for the Hodge structure on $H^2$ is given as the composition
\begin{equation}\label{eq: period map tori}
\begin{tikzcd}
\Def_{X_U} \arrow{dr}[swap]{\kp^1}\arrow{rr}{\kp^2} & & \Grass\left( \binom {n}{2}, \Wedge^2 V^*\right)   \\ 
  & \Grass(n, V^*)\arrow{ur}{h}[swap]{W\mapsto \bigwedge^2 W}   
\end{tikzcd}
\end{equation}

\begin{exam}\label{ex: period 2 torus}
Consider $2$-tori, that is, the case $n=2$ in the above diagram. Then  $\Grass\left( \binom {n}{2}, \Wedge^2 V^*\right) \isom \IP^5$ and $h$ is the Pl\"ucker embedding of the Gra{\ss}mannian $\Grass(2,4)$ as a quadric in $\IP^5$. Considering a $2$-torus $X_U$ as a  simple $\del\delbar$-complex symplectic manifold we thus recover Theorem \ref{thm: simple local torelli} in this case. Compare Section \ref{sect: computation tori} for a direct computation.
\end{exam}

To our suprise we could not track down a reference where embeddings of Gra{\ss}mannians as in \eqref{eq: period map tori} have been studied classically, so we give some indication how one might work out their geometry.

The Picard group of $\Grass\left(\binom {n}{2}, \Wedge^2 V^*\right)$ is generated by an ample line bundle $A$, which induces the Pl\"ucker embedding
$f\colon \Grass\left(\binom {n}{2}, \Wedge^2 V^*\right) \to \IP\left(\Wedge^{\binom {n}{2}}\Wedge^2 V^*\right)=\IP$, that is, $A =f^*H$, where $H$ is the hyperplane class in $\IP$.
\begin{lem}
Let $B$ be the ample generator of the Picard group of $\Grass(n, V^*)$ and let $g = f\circ h$. Then  $h^*A = g^*H=(n-1) B$. 
\end{lem}
\begin{proof}
Write $h^*A = g^*\ko_{\IP}(1)=m B$. We aim to prove  that $m = n-1$. We refer to \cite{Griffiths-Harris} for the basic theory of Gra{\ss}mannians and Schubert calculus. 

Let $C$ be the Schubert cycle dual to $B$ in the cohomology ring of $\Grass(n, V^*)$, then by the projection formula we have that
\[m = m B \cdot C =  A\cdot h_*C = H\cdot g_* C.\]
Recall that $V^*$ is $2n$-dimensional, so if we fix a complete flag $V_1 \subseteq \ldots \subseteq V_i \subseteq \ldots \subseteq V_{2n} = V^*$ with $\dim V_i = i$, then $C$ parametrises the $n$-dimensional subspaces of $V^*$ which are contained in $V_{n + 1}$ and contain $V_{n - 1}$. It is then easy to see that $C \simeq \IP^1$: once we fix a basis $\{ x_1, \ldots, x_{2n}\}$ for $V^*$, then $C$ is given by
\[\left\{ \operatorname{Span} \{ x_1, \ldots, x_{n - 1}, \alpha x_n + \beta x_{n + 1} \} \middle| (\alpha: \beta) \in \IP^1 \right\}.\]
The choice of our basis induces the standard basis of $\bigwedge^2 V^*$: $\{ x_i \wedge x_j \}$ with $1 \leq i < j \leq 2n$, and for $W \in C$ we have that $h(W) = \bigwedge^2 W$ is generated by $x_i \wedge x_j$ for $1 \leq i < j \leq n - 1$ and $x_k \wedge (\alpha x_n + \beta x_{n + 1})$ for $1 \leq k \leq n - 1$. Now, considering the larger Gra{\ss}mannian in its Pl\"ucker embedding $f$, either $g(C)$ is all contained in a hyperplane, or $g(C)$ cuts such hyperplane in $m$ points. As hyperplanes are defined by linear combination of Pl\"ucker coordinates, we choose the hyperplane defined by the vanishing of a single Pl\"ucker coordinate, the one corresponding to the choice of multi-indices $(i, j)$ with $1 \leq i < j \leq n$. So (up to permutations of the columns) the corresponding coordinate is the determinant of the matrix
\[\left( \begin{array}{c|c}
\id_{\binom{n - 1}{2}} & 0\\
\hline
0 & \alpha \cdot \id_{n - 1}
\end{array} \right).\]
So we see that this determinant vanishes only for $\alpha = 0$ of order $n - 1$, this means that $m = n - 1$ and so the result follows.
\end{proof}

To understand what happens on the level of global section we note that the maps $f$, $g$, and $h$ are equivariant under the natural 
$\mathrm{Gl}(V)$ action. Thus, on global sections we get the induced map of representations
\[\begin{tikzcd}H^0(\IP, H) = H^0(A) = \Wedge^{ \binom {n}{2}}\Wedge^2 V \rar{g^*} & H^0(B^{\tensor(n-1)}) = \IS_{(n-1, \dots, n-1, 0)} V, \end{tikzcd}\]
where the right hand side is, by the Borel--Weil Theorem,  the Weyl module (see \cite{Fulton-Harris}) of the given partition.
Since the representation $H^0(B^{\tensor(n-1)})$ is irreducible and the map is non-zero, the map is actually the projection onto a direct summand of $H^0(A)$, considered as a $\mathrm{Gl}(V)$-representation.

Thus, we can extend \eqref{eq: period map tori} to the diagram
{\small
\[
 \begin{tikzcd}
\Def_{X_U} \arrow{dr}[swap]{\kp^1}\arrow{rr}{\kp^2} & & \Grass\left( \binom {n}{2}, \Wedge^2 V^*\right)    \rar{|A|} & \IP\left(\Wedge^{\binom {n}{2}}\Wedge^2 V^*\right)=\IP\\
  & \Grass(n, V^*)\arrow{ur}{h}[swap]{W\mapsto \bigwedge^2 W}    \arrow{rr}[swap]{|(n-1)B|} && \IP(\IS_{(n-1, \dots, n-1, 0)} V)^*\arrow[hookrightarrow]{u}\\ 
\end{tikzcd}
\]
}
Note  that the image of $\Grass(n, V^*)$ in $\IP$ is not the intersection of the larger Gra{\ss}mannian with the linear subspace, as shown by Example \ref{ex: period 2 torus}, and that the codimension of the image of the period map becomes very large as $n$ grows.

\section{Further Examples and questions}

\subsection{Simple \texorpdfstring{$\del\delbar$}{del-delbar}-complex symplectic manifolds}
Unfortunately, there is a lack of exmples of simple $\del\delbar$-complex symplectic manifolds which are not K\"ahler. Basically the only example we know is the following.

\begin{exam}
Let $X$ be an irreducible holomorphic symplectic (=hyperk\"ahler) manifold of dimension $2n$ and assume we have a Lagrangian $P=\IP^n \subset X$. Then we can perform the Mukai-flop of $X$ at $P$ (see \cite[Ex.\ 21.7]{HuybrechtsHKM}) and get a holomorphic symplectic manifold $X'$ which is in class $\kc$ (hence satisfies the $\del\delbar$-Lemma) but does not need to be K\"ahler\footnote{An explicit example is in \cite{Yoshioka01}}.
\end{exam}

So it remains to raise some questions on this class of manifolds.
\begin{quest}
Is there a simple $\del\delbar$-complex symplectic manifold (with $b_1=0$), which is not in Fujiki's class $\kc$?
\end{quest}
 
\begin{quest}
Is every simply connected, simple, complex symplectic manifold in Fujiki's class $\kc$ birational to a hyperk\"ahler manifold (possibly after a small deformation)? 

Indeed one could use Theorem \ref{thm: simple local torelli} to find a small deformation of $X$ which has a rational $(1,1)$-class $\alpha$ such that $q_\sigma(\alpha)>0$. Can one then construct a K\"ahler current in this class?
\end{quest}

\begin{rem}
The examples constructed by Guan \cite{Guan95a, Guan95b} probably do not satisfy the  $\del\delbar$-Lemma, although there seems no written proof for that. 

Related constructions are given by Toma \cite{Toma01}.
\end{rem}

\subsection{Complex tori}\label{sect: computation tori}

We perform some computations on the Beauville--Bogomolov--Fujiki form on complex tori in general, and then make it explicit in dimension $2$ and $4$. The Dolbeault algebra (and the Dolbeault cohomology) of a complex torus $T$ of dimension $2n$ is freely generated by the $2n$ forms of type $(1, 0)$ induced by the coordinates on $\IC^{2n}$: we call them $x_1, \ldots, x_{2n}$ and let $\bar{x}_i$ be the $(0, 1)$-form conjugate to $x_i$.

A basis for the space of $(2, 0)$-forms is $x_i \wedge x_j$ with $1 \leq i < j \leq 2n$, a basis for the space of $(0, 2)$-forms is obtained by conjugation of this one, and finally a basis for the space of $(1, 1)$-forms is $x_i \wedge \bar{x}_j$ with $1 \leq i, j \leq 2n$.

Let
\[\sigma = \sum_{1 \leq i < j \leq 2n} \lambda_{ij} x_i \wedge x_j\]
be a $(2, 0)$-form: then it is ($d$-closed, $\delbar$-closed and) non-degenerate if and only if we have $\mu \neq 0$ in the expression $\sigma^n = \mu \cdot x_1 \wedge \ldots \wedge x_{2n}$ with
\[\mu = \sum \varepsilon(i^1_1, i^1_2, \ldots, i^n_1, i^n_2) \lambda_{i^1_1 i^1_2} \cdot \ldots \cdot \lambda_{i^n_1 i^n_2}\]
and the sum is over all the partitions of $\{1, \ldots, 2n\}$ in disjoint couples $\{i^1_1, i^1_2\}$, \dots, $\{i^n_1, i^n_2\}$ with $i^t_1 < i^t_2$ for all $1 \leq t \leq n$.

\begin{rem}
The expression for $\mu$ is homogeneous of degree $n$ in the coordinates $\lambda_{ij}$.
\end{rem}

In the same way, we can compute that
\[\sigma^{n - 1} = \sum_{1 \leq i < j \leq 2n} \nu_{ij} \cdot x_1 \wedge \ldots \wedge \hat{x}_i \wedge \ldots \wedge \hat{x}_j \wedge \ldots \wedge x_{2n},\]
with
\[\nu_{ij} = \sum \varepsilon(h^1_1, h^1_2, \ldots, h^{n - 1}_1, h^{n - 1}_2) \lambda_{h^1_1 h^1_2} \cdot \ldots \cdot \lambda_{h^{n - 1}_1 h^{n - 1}_2}\]
and the sum is over all the partitions of $\{1, \ldots, 2n\} \smallsetminus \{i, j\}$ in disjoint couples $\{h^1_1, h^1_2\}, \ldots, \{h^{n - 1}_1, h^{n - 1}_2\}$ with $h^t_1 < i^h_2$ for all $1 \leq t \leq n - 1$.

From now on we assume that $\sigma$ is a symplectic form (so $\mu \neq 0$) normalized in such a way that
\[\int_T (\sigma \bar{\sigma})^n = \mu \bar{\mu} \int_T \Vol = 1,\]
where $\Vol = x_1 \wedge \ldots \wedge x_{2n} \wedge \bar{x}_1 \wedge \ldots \wedge \bar{x}_{2n}$ is the usual volume form. Writing explicitly the symmetric bilinear form associated to $q_\sigma$ we find that its expression is
{\small
\[
 2 \langle \psi, \eta \rangle_\sigma = n \int_T (\sigma \bar{\sigma})^{n - 1} \psi \eta + (1 - n) \left( \int_T \sigma^{n - 1}\bar{\sigma}^n \psi \int_T \sigma^{n}\bar{\sigma}^{n - 1} \eta + \int_T \sigma^{n - 1}\bar{\sigma}^n \eta \int_T \sigma^{n}\bar{\sigma}^{n - 1} \psi \right).
\]
}
So we see that it is only a matter of bidegree that $\langle H^{2, 0}(T), H^{2, 0}(T) \rangle_\sigma = 0$ and $\langle H^{0, 2}(T), H^{0, 2}(T) \rangle_\sigma = 0$. Moreover, it follows from \cite{CattaneoTomassini16} that $H^{2, 0}(T) \oplus H^{0, 2}(T)$ and $H^{1, 1}(T)$ are orthogonal to each other, and so the matrix expressing the bilinear form has the shape:
\[
\renewcommand{\arraystretch}{1.2}
\begin{array}{c|ccc}
 & H^{2, 0}(T) &  H^{1, 1}(T) &  H^{0, 2}(T)\\
\hline
H^{2, 0}(T) & 0 & 0 & \\
H^{1, 1}(T) & 0 &  & 0\\
H^{0, 2}(T) &  & 0 & 0
\end{array}\]
So we need only to compute $\langle H^{2, 0}(T), H^{0, 2}(T) \rangle_\sigma$ and $\langle H^{1, 1}(T), H^{1, 1}(T) \rangle_\sigma$.

We begin with $\langle x_\alpha \wedge \bar{x}_\beta, x_\gamma \wedge \bar{x}_\delta \rangle_\sigma$: if $\alpha = \gamma$ or $\beta = \delta$ this pairing is $0$, otherwise
\[\begin{array}{rl}
2 \langle x_\alpha \wedge \bar{x}_\beta, x_\gamma \wedge \bar{x}_\delta \rangle_\sigma = & n \int_T (\sigma \bar{\sigma})^{n - 1} x_\alpha \wedge \bar{x}_\beta \wedge x_\gamma \wedge \bar{x}_\delta =\\
= & (-1)^e n \nu_{\min\{ \alpha, \gamma \}, \max\{ \alpha, \gamma \}} \bar{\nu}_{\min\{ \beta, \delta \}, \max\{ \beta, \delta \}} \int_T \Vol.
\end{array}\]
Hence
\[\langle x_\alpha \wedge \bar{x}_\beta, x_\gamma \wedge \bar{x}_\delta \rangle_\sigma = (-1)^e n \frac{\nu_{\min\{ \alpha, \gamma \}, \max\{ \alpha, \gamma \}} \bar{\nu}_{\min\{ \beta, \delta \}, \max\{ \beta, \delta \}}}{2 \mu \bar{\mu}},\]
where the exponent $e$ is determined as follows:
\[\begin{array}{c|cc}
e & \beta < \delta & \beta > \delta\\
\hline
\alpha < \gamma & \alpha + \beta + \gamma + \delta + 1 & \alpha + \beta + \gamma + \delta\\	
\alpha > \gamma & \alpha + \beta + \gamma + \delta & \alpha + \beta + \gamma + \delta + 1\\
\end{array}
\]

We now compute $\langle x_\alpha \wedge x_\beta, \bar{x}_\gamma \wedge \bar{x}_\delta \rangle_\sigma$:
\[\begin{array}{rl}
2 \langle x_\alpha \wedge x_\beta, \bar{x}_\gamma \wedge \bar{x}_\delta \rangle_\sigma = & n \int_T (\sigma \bar{\sigma})^{n - 1} x_\alpha \wedge x_\beta \wedge \bar{x}_\gamma \wedge \bar{x}_\delta +\\
 & + (1 - n) \int_T \sigma^{n - 1}\bar{\sigma}^n x_\alpha \wedge x_\beta \int_T \sigma^{n}\bar{\sigma}^{n - 1} \bar{x}_\gamma \wedge \bar{x}_\delta =\\
= & (-1)^{\alpha + \beta + \gamma + \delta} n \nu_{\alpha \beta} \bar{\nu}_{\gamma \delta} \int_T \Vol +\\
 & + (1 - n) \left( (-1)^{\alpha + \beta + 1} \bar{\mu} \nu_{\alpha \beta} \int_T \Vol \cdot (-1)^{\gamma + \delta + 1} \mu \bar{\nu}_{\gamma \delta} \int_T \Vol \right),
\end{array}\]
from which we deduce
\[\langle x_\alpha \wedge x_\beta, \bar{x}_\gamma \wedge \bar{x}_\delta \rangle_\sigma = (-1)^{\alpha + \beta + \gamma + \delta} \frac{\nu_{\alpha \beta} \bar{\nu}_{\gamma \delta}}{2 \mu \bar{\mu}}.\]

In the case of a $2$-dimensional and $4$-dimensional torus respectively, we have the following Gram matrix for the Beauville--Bogomolov--Fujiki form on $H^2$.

For a $2$-dimensional torus, the situation is quite clear since we have
\[\begin{array}{l}
H^{2, 0}(T) = \operatorname{Span} \{ x_1 \wedge x_2 \},\\
H^{1, 1}(T) = \operatorname{Span} \{ x_1 \wedge \bar{x}_1, x_1 \wedge \bar{x}_2, x_2 \wedge \bar{x}_1, x_2 \wedge \bar{x}_2 \},\\
H^{0, 2}(T) = \operatorname{Span} \{ \bar{x}_1 \wedge \bar{x}_2 \},
\end{array}\]
and the expression for the BBF-bilinear-form does not depend on $\sigma$:
\[\langle \psi, \eta \rangle = \frac{1}{2} \int_T \psi \wedge \eta.\]
So, choosing any $(2, 0)$-form $\sigma = \mu x_1 \wedge x_2$ with $\int_T \sigma \bar{\sigma} = \mu \bar{\mu} \int_T x_1 \wedge x_2 \wedge \bar{x}_1 \wedge \bar{x}_2 = 1$, we have the Gram matrix
\[\left( \begin{array}{c|cccc|c}
0 & 0 & 0 & 0 & 0 & \frac{1}{2 \mu \bar{\mu}}\\
\hline
0 & 0 & 0 & 0 & -\frac{1}{2 \mu \bar{\mu}} & 0\\
0 & 0 & 0 & \frac{1}{2 \mu \bar{\mu}} & 0 & 0\\
0 & 0 & \frac{1}{2 \mu \bar{\mu}} & 0 & 0 & 0\\
0 & -\frac{1}{2 \mu \bar{\mu}} & 0 & 0 & 0 & 0\\
\hline
\frac{1}{2 \mu \bar{\mu}} & 0 & 0 & 0 & 0 & 0
\end{array} \right).\]

On the $4$-dimensional case, we have
\[\begin{array}{l}
H^{2, 0}(T) = \operatorname{Span} \left\{ \begin{array}{ccc}
x_1 \wedge x_2 & x_1 \wedge x_3 & x_1 \wedge x_4\\
x_2 \wedge x_3 & x_2 \wedge x_4 & x_3 \wedge x_4
\end{array} \right\},\\
H^{1, 1}(T) = \operatorname{Span} \left\{ \begin{array}{cccc}
x_1 \wedge \bar{x}_1 & x_1 \wedge \bar{x}_2 & x_1 \wedge \bar{x}_3 & x_1 \wedge \bar{x}_4\\
x_2 \wedge \bar{x}_1 & x_2 \wedge \bar{x}_2 & x_2 \wedge \bar{x}_3 & x_2 \wedge \bar{x}_4\\
x_3 \wedge \bar{x}_1 & x_3 \wedge \bar{x}_2 & x_3 \wedge \bar{x}_3 & x_3 \wedge \bar{x}_4\\
x_4 \wedge \bar{x}_1 & x_4 \wedge \bar{x}_2 & x_4 \wedge \bar{x}_3 & x_4 \wedge \bar{x}_4
\end{array} \right\},\\
H^{0, 2}(T) = \operatorname{Span} \left\{ \begin{array}{ccc}
\bar{x}_1 \wedge \bar{x}_2 & \bar{x}_1 \wedge \bar{x}_3 & \bar{x}_1 \wedge \bar{x}_4\\
\bar{x}_2 \wedge \bar{x}_3 & \bar{x}_2 \wedge \bar{x}_4 & \bar{x}_3 \wedge \bar{x}_4
\end{array} \right\}.
\end{array}\]
So we see that if we consider a $(2, 0)$-form
\[\sigma = \lambda_{12} x_1 \wedge x_2 + \lambda_{13} x_1 \wedge x_3 + \lambda_{14} x_1 \wedge x_4 + \lambda_{23} x_2 \wedge x_3 + \lambda_{24} x_2 \wedge x_4 + \lambda_{34} x_3 \wedge x_4,\]
then it is non-degenerate if and only if
\[\sigma^2 = \underbrace{2(\lambda_{12} \lambda_{34} - \lambda_{13} \lambda_{24} + \lambda_{14} \lambda_{23})}_{\mu} \cdot x_1 \wedge x_2 \wedge x_3 \wedge x_4 \neq 0.\]

Now, since we are dealing with $4$ indices it follows that $\nu_{ij}$ is one of the coefficient $\lambda$, to be precise it is the one corresponding to the complement of $\{i, j\}$.

An explicit computation of the Gram matrix yields to
\[\left( \begin{array}{c|c|c}
0 & 0 & X\\
0 & Y & 0\\
X^t & 0 & 0
\end{array} \right)\]
where
\[X = \frac{1}{2 \mu \bar{\mu}} \left( \begin{array}{cccccc}
\lambda_{34} \bar{\lambda}_{34} & -\lambda_{34} \bar{\lambda}_{24} & \lambda_{34} \bar{\lambda}_{23} & \lambda_{34} \bar{\lambda}_{14} & -\lambda_{34} \bar{\lambda}_{13} & \lambda_{34} \bar{\lambda}_{12}\\

-\lambda_{24} \bar{\lambda}_{34} & \lambda_{24} \bar{\lambda}_{24} & -\lambda_{24} \bar{\lambda}_{23} & -\lambda_{24} \bar{\lambda}_{14} & \lambda_{24} \bar{\lambda}_{13} & -\lambda_{24} \bar{\lambda}_{12}\\

\lambda_{23} \bar{\lambda}_{34} & -\lambda_{23} \bar{\lambda}_{24} & \lambda_{23} \bar{\lambda}_{23} & \lambda_{23} \bar{\lambda}_{14} & -\lambda_{23} \bar{\lambda}_{13} & \lambda_{23} \bar{\lambda}_{12}\\

\lambda_{14} \bar{\lambda}_{34} & -\lambda_{14} \bar{\lambda}_{24} & \lambda_{14} \bar{\lambda}_{23} & \lambda_{14} \bar{\lambda}_{14} & -\lambda_{14} \bar{\lambda}_{13} & \lambda_{14} \bar{\lambda}_{12}\\

-\lambda_{13} \bar{\lambda}_{34} & \lambda_{13} \bar{\lambda}_{24} & -\lambda_{13} \bar{\lambda}_{23} & -\lambda_{13} \bar{\lambda}_{14} & \lambda_{13} \bar{\lambda}_{13} & -\lambda_{13} \bar{\lambda}_{12}\\

\lambda_{12} \bar{\lambda}_{34} & -\lambda_{12} \bar{\lambda}_{24} & \lambda_{12} \bar{\lambda}_{23} & \lambda_{12} \bar{\lambda}_{14} & -\lambda_{12} \bar{\lambda}_{13} & \lambda_{12} \bar{\lambda}_{12}
\end{array} \right),\]
and $Y$ has to be computed.

\subsection{Kodaira surface}\label{ex: Kodaira surface}
We quickly discuss an example that does not satisfy the $\del\delbar$-Lemma, which was also used in Example \ref{exam: kodaira fibration lambda}.

Consider the standard Kodaira surface, i.e.\ the quotient space of the group
\[G = \left\{ \left( \begin{array}{ccc}
1 & \bar{z}_1 & z_2\\
0 & 1 & z_1\\
0 & 0 & 1
\end{array} \right) \middle| z_1, z_2 \in \IC \right\}\]
by its lattice $\Gamma$ consisting of matrices with entries in $\IZ[\sqrt{-1}]$: so $X = \Gamma \backslash G$. There are the following $(1, 0)$-forms: $\omega_1 = dz_1$ and $\omega_2 = dz_2 - \bar{z}_1 dz_1$, which together to their complex conjugates give all the $1$-forms.

There is up to scalars, only one $(2, 0)$-form, namely $\omega_1 \wedge \omega_2$, which is also $d$-closed. Then a complex symplectic structure on $X$ is $\sigma = \mu \omega_1 \wedge \omega_2$ for $\mu \neq 0$. We will assume $\sigma$ normalized, so that
\[\int_X \sigma \bar{\sigma} = \mu \bar{\mu} \int_X \underbrace{\omega_1 \wedge \omega_2}_{vol} = 1.\]
Observe that since we are on a surface, then
\[\langle \psi, \eta \rangle = \frac{1}{2} \int_X \psi \wedge \eta.\]

We can also see that
\[\begin{array}{l}
H^1_{dR}(X, \IC) = \operatorname{Span} \{ \omega_1, \bar{\omega}_1, \omega_2 + \bar{\omega}_2 \}\\
H^2_{dR}(X, \IC) = \operatorname{Span} \{ \omega_1 \wedge \omega_2, \omega_1 \wedge \bar{\omega}_2, \omega_2 \wedge \bar{\omega}_1, \bar{\omega}_1 \wedge \bar{\omega}_2\}
\end{array}\]
while for the Dolbeault cohomology we have
\[\begin{array}{l}
H^{1, 0}_{\delbar}(X) = \operatorname{Span} \{ \omega_1 \}\\
H^{0, 1}_{\delbar}(X) = \operatorname{Span} \{ \bar{\omega}_1, \bar{\omega}_2 \}\\
H^{2, 0}_{\delbar}(X) = \operatorname{Span} \{ \omega_1 \wedge \omega_2 \}\\
H^{1, 1}_{\delbar}(X) = \operatorname{Span} \{ \omega_1 \wedge \bar{\omega}_2, \omega_2 \wedge \bar{\omega}_1 \}\\
H^{0, 2}_{\delbar}(X) = \operatorname{Span} \{ \bar{\omega}_1 \wedge \bar{\omega}_2 \}.
\end{array}\]
This shows that $X$ does not satisfy the $\del \delbar$-Lemma, looking to its $1$-forms, but its second cohomology group splits into the direct sum of types and conjugation is an isomorphism.

But then the Beauville--Bogomolov--Fujiki form has Gram matrix (with respect to the basis for $H^2_{dR}(X, \IC)$ above)
\[\frac{1}{2 \mu \bar{\mu}} \left( \begin{array}{c|cc|c}
0 & 0 & 0 & 1\\
\hline
0 & 0 & 1 & 0\\
0 & 1 & 0 & 0\\
\hline
1 & 0 & 0 & 0
\end{array} \right),\]
showing that the Beauville--Bogomolov--Fujiki quadric of a Kodaira surface is smooth irreducible.

\bibliographystyle{alpha}
\bibliography{symplecticdeldelbar}

\end{document}